\documentclass[oneside]{amsart}
\usepackage[latin1]{inputenc}
\usepackage[T1]{fontenc}
\usepackage{amstext}
\usepackage{amsmath, amssymb}
\usepackage[english]{babel}
\usepackage{amsfonts}

\newcommand{\R}{\mathbb{R}}

\newcommand{\N}{\mathbb{N}}
\newcommand{\K}{\mathbb{K}}

\newcommand{\E}{\mathbb{E}}
\newcommand{\B}{\mathbb{B}}
\newcommand{\Z}{\mathbb{Z}}
\newcommand{\Q}{\mathbb{Q}}
\newcommand{\pp}{\mathbb{P}}
\newcommand\kB{\mathcal{B}}

\newcommand\kP{\mathcal{P}}
\newcommand\kF{\mathcal{F}}

\newtheorem {lem} {Lemma} [section]
\newtheorem {prop} {Proposition} [section]
\newtheorem {theo} {Theorem} [section]
\newtheorem {cor} {Corollary} [section]

\newcommand\la{\lambda}

\title{Poisson boundary of $GL_d(\Q)$}

\author{Sara Brofferio}
\author{Bruno Schapira}
\address{D\'epartement de Math\'ematiques, B\^at. 425, Universit\'e Paris-Sud 11, F-91405 Orsay, cedex, France. }
\email{sara.brofferio@math.u-psud.fr}
\email{bruno.schapira@math.u-psud.fr}

\begin{document}

%\today

\begin{abstract}  We construct the Poisson boundary for a random walk supported by the general
linear group on the rational numbers as the product of flag manifolds over the $p$-adic fields. To this purpose, we prove a law of large numbers using the Oseledets' multiplicative ergodic theorem. The only assumption we need is some moment condition on the measure governing the jumps of the random walk, but no irreducibility hypothesis is made. 
\end{abstract}



\maketitle

\section{Introduction}

The Poisson boundary of a group endowed with some measure $\mu$, describes the asymptotic behavior of the random walk with step law $\mu$. In the same time it gives a representation of bounded harmonic functions (see for instance \cite{Fur} for a survey on this topic).  
There are now many results on Poisson boundary of groups of matrices (see for instance \cite{Az, F, GR, Kai00, KV, Led2, Rau} for some of the main results in this field, and \cite{BL, F2} for some surveys).

Here we consider more specifically groups of matrices with rational coefficients, which were already considered in our previous works \cite{Bro, Sch} for subgroups of triangular matrices. The novelty in the rational case, in comparison with standard results on real matrices, is that to describe the Poisson boundary, one has now to consider all possible embeddings of the rational field in the $p$-adic fields, and the Poisson boundary is then a product of "local parts", one for each prime number $p$ (see Theorem \ref{theo1} below for a precise statement of our result).
This phenomenon was already observed in \cite{Kai91} for the group of affine transformations with dyadic coefficients, and is very similar to some result proved in \cite{BS} in an adelic setting. It should be noticed also that we do not need any hypothesis on the support of the measure $\mu$.

Denote by $\kP^*$ the set of prime numbers and let $\kP= \kP^*\cup \{\infty\}$. For $p\in \kP^*$, denote by $\Q_p$ the field of $p$-adic numbers, and set by convention $\Q_\infty=\R$.

If $\mu$ is a probability measure on $GL_d(\Q_p)$ with finite logarithmic moment, i.e.
$$\int \left(\ln^+ ||g||_p+\ln ^+||g^{-1}||_p\right)\ d\mu(g)<+\infty,$$
the associated Lyapunov exponents are the real numbers $\la_1(p)\ge \dots\ge \la_d(p)$ such that
$$\sum_{i=1}^k \la_i(p) = \lim_{n\to +\infty}\frac{1}{n}\int \ln ||\bigwedge^k g||_p\ d\mu^{*n}(g),$$
where $\mu^{*n}$ denotes the $n$-fold convolution of $\mu$ and $\wedge$ the exterior product.
Denote by $P_p$ the parabolic subgroup of $GL_d(\Q_p)$ consisting of matrices $(p_{i,j})$ with $p_{i,j}=0$, when $\lambda_i(p)<\lambda_j(p)$, and let $B_p:=GL_d(\Q_p)/P_p$ be the associated flag manifold.

\noindent The main result of this paper is the following:

\begin{theo}
\label{theo1}
Let $\mu$ be a probability measure on $GL_d(\Q)$ such that
$$\sum_{p\in \kP} \int \left(\ln^+ ||g||_p+\ln ^+||g^{-1}||_p\right)\ d\mu(g)<+\infty.$$
Then there exists a unique probability measure $\nu$ on the space
$$\mathbb{B}:=\prod_{p\in \kP} B_p,$$
such that $(\mathbb{B},\nu)$ is the Poisson boundary of $(GL_d(\Q),\mu)$.
\end{theo}
This theorem unifies and generalizes several results on Poisson boundary of rational matrices groups, known up to now. In particular, it has been  proved separately by F. Ledrappier \cite{Led2} and V. Kaimanovich \cite{Kai85} that the Poisson boundary of a random walk supported by $SL_d(\Z)$  is the real flag manifold $B_\infty$. This results is contained in Theorem \ref{theo1} because, in this case, for all $p\neq\infty$, the associated Lyapunov exponents are all equal to zero, thus $B_p$ is trivial. Furthermore since Theorem \ref{theo1} does not require any irreducibility condition, it also applies to the case of rational affine group and to rational triangular matrices previously threaded by the authors \cite{Bro, Sch}.

We would like to remark that for general number fields (i.e. finite extensions of $\Q$) a similar result can be proved by adapting our methods (see in \cite{Sch} hints to possible generalization).

Due to its generality, our result does not say much about $\nu$ and its support. In particular it is not true that the restriction of $\nu$ to each $B_p$ has always full support. For instance if $\mu$ is supported on the subgroup of upper triangular matrices, we know \cite{Bro, Sch} that $\nu$ charges only one Bruhat cell of each $B_p$. But even this is not optimal since $\mu$ could be supported on diagonal matrices and with all Lyapunov exponents distinct, but in this case the Poisson boundary would be trivial (one point).  However irreducibility hypothesis  can give information on the support of $\mu$. We have for instance the following triviality criterion:
\begin{cor}\label{cor-triv} Let $p\in \kP$.
If $\lambda_1(p)=\lambda_d(p)$, then the projection of $\nu$ on $B_p$  is trivial.

Conversely, if the projection of $\nu$ on $B_p$  is trivial and no proper subspace of $\Q_p^d$ is fixed by the support of $\mu$, then $\lambda_1(p)=\lambda_d(p)$.
\end{cor}

There exists several results in the literature to decide whether the real Lyapunov exponents are all equal $\lambda_1(\infty)=\lambda_d(\infty)$. For instance, under irreducibility hypothesis this is  equivalent to ask that the closed subgroup generated by $\mu$ in $GL_d(\R)$ is amenable \cite{Gui80}. For other references and results on product of real random matrices, see  also \cite{BL}. It seem very likely that similar results hold on $p$-adic setting.

A different question that is still open is to understand the behavior of the measure $\nu$ on the product of the $p$-flag manifolds, and not only of its projection on each $B_p$. For instance: does $\nu$ charge  the whole product or  is it supported by some  "diagonal" sub-set? Is there some sort of correlation among the different $p$-adic components?

 The main tool of the proof of  Theorem \ref{theo1} is to produce, using the multiplicative ergodic theorem of Oseledets, a law of large numbers for random walks on $GL_d(\Q_p)$ (not necessarily with rational coefficients, see Proposition \ref{LLN}). Notice that such result on $GL_d(\R)$ or on the affine group over $\Q_p$ was already known (see \cite{Kai87} and \cite{CKW} respectively). The Lyapunov exponents give the speeds of convergence and the boundary limit of the random walk on $B_p$ the directions. This is done in Section 3, where we also use this result to prove that $B_p$ and $\B$ are $\mu$-boundaries.

In Section 4, we use entropy criterion due to Kaimanovich to establish the maximality of $(\B,\nu)$ and prove the main theorem and its corollary.

We notice that our strategy is very similar in spirit to that used by Karlsson and Margulis in \cite{KM} in a slightly different setting. But here the proof is more direct, since we can use Oseledets theorem, and we do not need to identify the Poisson boundary with the geometric boundary of some non-positively curved metric space.

The authors would like to thank Uri Bader for suggesting them  the problem. They are also grateful to Fran{\c c}ois Ledrappier and Anders Karlsson for useful advices and references.

\section{Preliminaries}
\subsection{General linear group over $\Q_p$} If $\K$ is a field, we denote by $GL_d(\K)$ the group of invertible matrices of size $d$ with coefficients in $\K$. We denote by $e$ the identity matrix.

For $p\in \kP$ and $v=(v_1,\dots,v_d) \in \Q_p^d$, we set
$$|v|_p=\max_i |v_i|_p,\mbox{ if } p\neq \infty \quad\mbox{ and }\quad |v|_\infty=\sqrt{\sum_i |v_i|_\infty^2},$$
and if $g\in GL_d(\Q_p)$ we set
$$||g||_p=\sup_{|v|_p=1} |gv|_p.$$

For any $p\in\kP$ and $g,h\in GL_d(\Q_p)$ set
$$d_p(g,h)= \ln^+ ||g^{-1}h||_p + \ln^+||h^{-1}g||_p,$$
where $\ln^+$ denotes the positive part of the function $\ln $. It is easily  checked that $d_p$ is symmetric and satisfies the triangular inequality. It is not a distance since the set of $g\in GL_d(\Q_p)$ such that $d_p(e,g)=0$ is the compact subgroup of linear isometries of $\Q_p^d$.
Furthermore $d_p$ is left-invariant:
$$d_p(\gamma g, \gamma h) = d_p(g,h), $$
for all $g,h,\gamma \in GL_d(\Q_p)$.

For all $g,h\in GL_d(\Q)$, let
$$d(g,h)=\sum_{p\in\kP} d_p(g,h).$$
This define a left-invariant pseudometric on $GL_d(\Q)$.

\subsection{The flag manifold}
\label{jis}

For each $p\in \kP$ fix the sequence of Lyapunov exponents $\la_1(p)\ge \dots\ge \la_d(p)$. The associated parabolic sub-group is
$$P_p=\left\{(p_{i,j})\in GL_d(\Q_p) \mid p_{i,j}=0\mbox{ if }\lambda_i(p)<\lambda_j(p)\right\}.$$
The flag manifold  $B_p:=GL_d(\Q_p)/P_p$  is then  a compact separable $GL_d(\Q_p)$-space.

We mention that there is a one to one map between $B_p$ and the space of flags, viewed as the set of imbedded sequences of sub-spaces of $\Q_p^d$ of fixed dimensions. In fact
$$B_p=\left\{(V_1,\ldots,V_r)\mid V_1\leq \cdots \leq V_r=Q_p^d, \quad \dim (V_i)=j_i\quad \forall i\le r\right\},$$
where $r$ is the number of distinct values taken by $\la_1(p),\dots,\la_d(p)$, and $j_1,\dots,j_r$ are defined inductively by $j_r=d$ and $j_{i-1}=\max\{j<j_i\mid \la_j(p)>\la_{j_i}(p)\}$, for $2\le i\le r$. To see the correspondence between $B_p$ and this space of flags, observe that $GL_d(\Q_p)$ acts transitively on the flags
and that the parabolic subgroup $P_p$ is the stabilizer of the element $(E_1,\dots,E_r)$, where for all $i$,
$E_i$ is the vector space generated by the first $j_i$ vectors of the canonical basis.

Let $$\mathbb{B}:=\prod_{p\in \kP} B_p,$$ equipped with the product topology. With the natural diagonal action, $\B$ is a compact separable $GL_d(\Q)$-space.

\subsection{Random walk and $\mu$-boundaries} Let $\mu$ be a probability measure on a locally compact separable
group  $G$.
Let
$$(\Omega,\pp):=(G,\mu)^{\otimes \N},$$
be the product of $\N$ independent copies of $(G,\mu)$ (here $\N$ is the set of strictly positive integers).
If $w=(w_i,i\ge 1) \in \Omega$, the \emph{random walk} is the process defined by
$$x_n:=w_1\dots w_n \quad \forall n\ge 1\quad\mbox{and}\quad x_0:=e.$$
Observe that under $\pp$, for any fixed $n$, the law of  $x_n$ is $\mu^{*n}$, the $n$-th convolution power of $\mu$.

Assume that $B$ is a compact separable space, endowed with a probability
measure $\nu$ and a continuous action of $G$. We say that $\nu$ is
\emph{$\mu$-stationary} (also known as $\mu$-invariant or $\mu$-harmonic), if
$$\mu *\nu:=\int_{G} (g\nu)\ d\mu(g)=\nu,$$ where for
all $g\in G$, $g\nu$ is defined by
$$g\nu(f)=\int_{B}f(gz)\ d\nu(z),$$ for all continuous
functions $f$. In this case, according to Furstenberg \cite{F,F2}, we
say that $(B,\nu)$ is a \emph{$\mu$-boundary} if, $\pp$-almost
surely $x_n\nu$ converges  weakly to a Dirac measure.

A $\mu$-boundary $(B,\nu)$ is naturally associated to  a measurable function  $\mathbf{b}=\mathbf{b}_B: \Omega\to B$ defined by
\begin{equation}\label{eq-mu-bound}
\lim_{n\to +\infty} x_n \nu = \delta_{\mathbf{b}(w)}.
\end{equation}
Then $\nu$ is the image of $\pp$ under $\mathbf{b}$.

Denote by $\theta$ the shift transformation on $\Omega$: if $w=(w_i,i\ge 1) \in \Omega$, then
$$(\theta w)_i=w_{i+1} \quad i\ge 1.$$
The measure $\pp$ is $\theta$-invariant and it is easy to see that the function defined in \eqref{eq-mu-bound} satisfies
 $$w_1\mathbf{b}(\theta w) = \mathbf{b}(w).$$
This property characterizes functions that arise from $\mu$-boundaries, as follows from this known result (see for instance \cite{Kai00,Led2}):
\begin{prop}\label{prop-mu-bound} Let $B$ be a compact separable $G$-space and let $\textbf{b}:\Omega \to B$ be a measurable map, such that $\pp$-a.s. we have $w_1\mathbf{b}(\theta w) = \mathbf{b}(w)$. Let $\nu$ be the law of $\mathbf{b}$. Then $(B,\nu)$ is a $\mu$-boundary.
\end{prop}
\begin{proof} We give here a proof for sake of completeness. By using the invariance of $\pp$ by $\theta$ and the hypothesis on the map $b$ we get for every continuous functions $f$ on $B$
\begin{eqnarray*}
\nu (f) &=& \int_\Omega f(\mathbf{b}(w))\ d\pp(w)\\
            &=& \int_\Omega f(w_1 \mathbf{b}(\theta w))\ d\pp(w)\\
            &=& \int_G\left(\int_\Omega f(w_1 \mathbf{b}(w'))\ d\pp(w')\right)d\mu(w_1)=\mu*\nu(f),
\end{eqnarray*}
proving that $\nu$ is $\mu$-stationary.
The hypothesis on $\mathbf{b}$ also shows that for any continuous function $f$, the sequence
$$M_n(w):= x_n\cdot \nu (f) \quad n\ge 1,$$
is a bounded martingale. Thus this sequence converges a.s and in $L^1$ toward some limit, say $\nu_\infty^w(f)$. Since $B$ is separable, this defines actually a random measure $\nu^w_\infty$ on $B$, which is the weak limit of $x_n\cdot\nu$, $n\ge 1$. Observe now that for all $k\ge 1$,
$$\nu_\infty^w= x_k\nu_\infty^{\theta^kw} \quad \pp-a.s.$$
Moreover $\E[\nu_\infty^w]=\nu$. Observe also that the Dirac measure $\delta_{\mathbf{b}(w)}$ has the same properties. As a consequence for any $k\ge 1$, any Borel subsets $O_1,\dots,O_k \subset G$ and $U\subset B$,
\begin{eqnarray*}
\pp\nu_\infty^w[O_1\times \dots \times O_k \times U] & = & \int_\Omega \nu_\infty^w(U) 1(w_1\in O_1,\dots,w_k\in O_k)\ d\pp(w)\\
 & = &  \int_\Omega w_1\dots w_k \nu_\infty^{\theta^kw}( U) 1(w_1\in O_1,\dots,w_k\in O_k)\ d\pp(w)\\
  & = & \int_\Omega \nu((w_1\dots w_k )^{-1} U) 1(w_1\in O_1,\dots,w_k\in O_k)\ d\pp(w).
\end{eqnarray*}
For the same reason
$$\pp\delta_{\mathbf{b}(w)}[O_1\times \dots \times O_k \times U]=\int_\Omega \nu((w_1\dots w_k )^{-1} U) 1(w_1\in O_1,\dots,w_k\in O_k)\ d\pp(w).$$
Thus the two measures $\pp \nu_\infty^w$ and $\pp\delta_{\mathbf{b}(w)}$ defined on $\Omega \times B$ coincide on $\kF_k:=\sigma(w_1,\dots,w_k)\vee \kB$, for all $k\ge 1$, where $\kB$ denotes the Borel sets of $B$.
Since the filtration $(\kF_k)_{k\ge 0}$ generates the $\sigma$-algebra of $\Omega\times B$ on which are defined these measures, they are equal, proving that $\nu_\infty^w$ is well a Dirac measure. This concludes the proof of the proposition.
\end{proof}

\subsection{Poisson boundary and asymptotic entropy}
The Poisson boundary $(\B,\nu)$ is defined as the maximal $\mu$-boundary,
i.e. it is the $\mu$-boundary such that any other $\mu$-boundary is one of its measurable $G$-equinvariant quotients. A classical problem is to decide weather a space, that is known to be a $\mu$-boundary, is in fact the maximal one.

For countable groups, there exists powerful techniques based on the estimation of the entropy introduced
by Kaimanovich and Vershik \cite{KV} and Derrienic \cite{Der} and further developed by  Kaimanovich (see \cite{Kai00} for details). Suppose that the measure $\mu$ has finite entropy:
$$H(\mu):=-\sum_{g\in G} \mu(g)\ln \mu(g)<\infty.$$
If $(B,\nu)$ is a $\mu$-boundary and $z\in B$, it is possible to define  the law $\pp^z$ of $w\in \Omega$ conditioned by $\mathbf{b}(w)=z$. Then for $n\ge 0$, $\pp_n^z$ denotes the law of $x_n$ under  $\pp^z$, i.e.
$$\pp_n^z(g)=\pp^z(x_n=g)=\pp(x_n=g \mid \mathbf{b}(w)=z).$$
The conditional asymptotic entropy $h^z$ is defined by
$$h^z:=-\lim_{n\to +\infty} \frac{\ln \pp_n^z(x_n)}{n}\quad \pp^z-a.s.$$
Then $(B,\nu)$ is the Poisson boundary if, and only if,  $h^z$ is equal to zero for $\nu$-almost every $z$.

\section{Law of large numbers and $\mu$-boundaries for $GL_d(\Q_p)$}

In this section we can assume $\mu$ to be a probability measure on $GL_d(\Q_p)$, not necessarily supported on matrices with rational coefficients. We are going to show that, under first moment hypothesis, the random walk on $GL_d(\Q_p)$ satisfies a strong law of large numbers, in which the "speeds" of the drift are given by the Lyapunov exponents and   the "directions" are given by an element of the associated flag manifold $B_p$. This approach was introduced by V. Kaimanovich in \cite{Kai87} for semisimple Lie groups, as a group-geometrical version of the classical multiplicative ergodic theorem of Oseledets (see also \cite{KM} and \cite{KL}).

Related to this result, we will see that $B_p$, endowed with the law of the "direction", is a $\mu$-boundary for the random walk.

\subsection{Oseledets' theorem and law of large numbers}
If $p=\infty$, let $\Lambda_n =\Lambda_n(\infty)$ be the diagonal matrix of $GL_d(\R)$ with coefficients
\begin{equation}\label{eq-lambda-R}
  (\Lambda_n)_{i,i}:=e^{n \lambda_i(\infty) } \quad \forall i\le d.
\end{equation}
If $p\in \kP^*$, let $\Lambda_n=\Lambda_n(p)$ be the diagonal matrix with coefficients
\begin{equation}\label{eq-lambda-Qp}(\Lambda_n)_{i,i}:=p^{-\left[\frac{n\lambda_i(p)}{\ln p}\right]} \quad \forall i\le d;\end{equation}
where $[\cdot]$ is the integer part.
In such a way, $\Lambda_n$  has rational entries whose $p$-norms are close to the $e^{n \lambda_i (p)}$'s.
\begin{prop}
\label{LLN} Assume that $\int d_p(e,g)\, d\mu(g)<+\infty$. Then there exists a measurable map $$\mathbf{b}=\mathbf{b}_p:\Omega\to B_p,$$ such that $\pp$-almost surely  $\mathbf{b}(w)$ is the unique element of  $B_p$ such that
\begin{equation}\label{eq-LLN}
\lim_{n\to +\infty} \frac{1}{n} d_p\left(x_n,b\Lambda_n\right)=0,
\end{equation}
for any $b$ in the class of $\mathbf{b}(w)$.
\end{prop}

To prove this proposition we use the following lemma that translates Oseledets' Theorem in our setting:

\begin{lem}
\label{oselbis} Assume that $\int d_p(e,g)\, d\mu(g)<+\infty$. Then there exists a measurable map $$\mathbf{b}=\mathbf{b}_p:\Omega\to B_p,$$ such that $\pp$-almost surely
\begin{equation}\label{eq-oselbis}
\lim_{n\to +\infty} \frac{1}{n} \ln ||x_n^{-1} b\Lambda_n||_p=0,
\end{equation}
for any $b$ in the class of $\mathbf{b}(w)$.
\end{lem}
\begin{proof}
 Let us first recall the multiplicative ergodic theorem, first proved by V.I. Oseledets \cite{Ose} for real matrices and generalized by M.S. Raghunathan \cite{Rag} to matrices on local fields. It says that $\pp$-a.s. there exists a filtration of subspaces of $\Q_p^d$,   $\{0\}=V^0(w)\subset V^1(w)\subset \dots \subset V^r(w)=\Q_p^d$, such that
\begin{enumerate}
\item[(i)] The map $w\to V^i(w)$ is measurable for all $i\le r$.
\item[(ii)] For all $1\le i \le r$, we have $v\in V^i(w)\setminus V^{i-1}(w)$ if, and only if,
$$\lim_{n\to +\infty} \frac{1}{n} \ln ||w_n^{-1}\dots w_1^{-1} v||_p=-\la_{j_i}(p),$$
\end{enumerate}
where $r$ and $j_1,\dots,j_r$ are defined as in section \ref{jis}.

Suppose that $p\neq \infty$ (the real case is treated analogously).
Denote by $(e_1,\dots,e_d)$ the canonical basis of $\Q_p^d$. Consider a matrix $b\in GL_d(\Q_p)$ such that for all $i\le r$, $b$ sends the family $(e_1,\ldots,e_{j_i})$ into a basis of $V^i(w)$. Then $b=[v_1|\cdots|v_d]$, where for all
$i\le r$, $(v_1,\ldots,v_{j_i})$ is a basis of $V^i(w)$. Observe that
$$ x_n^{-1} b\Lambda_n=\left [x_n^{-1} v_1 p^{-\left[\frac{n\lambda_1(p)}{\ln p}\right]}\right|\quad\cdots \quad \left| x_n^{-1} v_d p^{-\left[\frac{n\lambda_d(p)}{\ln p}\right]}\right].$$
Then
$$ \max_{k=1,\ldots, d} \left( p^{\left[\frac{n\lambda_{k}(p)}{\ln p}\right]}|x_n^{-1} v_k |_p \right) \leq  \left\|x_n^{-1} b\Lambda_n\right\|_p \leq d\,\max_{k=1,\ldots, d} \left(p^{\left[\frac{n\lambda_{k}(p)}{\ln p}\right]} |x_n^{-1} v_k |_p \right).$$
Then, since $w_n^{-1}\dots w_1^{-1}=x_n^{-1}$, by (ii)
$$\lim_{n\to +\infty} \frac{1}{n} \ln ||x_n^{-1} b\Lambda_n||_p=0.$$
To conclude the proof, just observe that two matrices $b_1$ and $b_2$ give two bases of the same filtration $\{V_i(w)\}_i$ if, and only if, $b_1^{-1}b_2$ is in the group $P_p$, thus such matrix $b$ can be identified with an element of $B_p$.
\end{proof}

\begin{proof}[Proof of Proposition \ref{LLN}]
   Let $\widetilde{x}_n=(w_1^t)^{-1}\cdots (w_n^t)^{-1}=(x_n^t)^{-1}$ be the random  walk of law $\widetilde{\mu}$, image of $\mu$ under the map $g\mapsto (g^t)^{-1}$.
   Then the Lyapunov exponents associated to $\widetilde{\mu}$ are
   $$\widetilde{\lambda}_i=-\lambda_{d-i}.$$
   Let $\widetilde{\Lambda}_n$ be the diagonal matrix constructed with the exponents $\widetilde{\lambda}_i$ as in \eqref{eq-lambda-Qp}. For $\pp$-almost all $w$ there exists a $\widetilde{b}\in GL_d(\Q_p)$ such that:
   $$\lim_{n\to +\infty} \frac{1}{n} \ln ||x_n^t \widetilde{b}\widetilde{\Lambda}_n||_p=0.$$
Consider the matrix $s=\left[
                                 \begin{array}{ccc}
                                   0 & \cdots & 1 \\
                                   \vdots & \cdots & \vdots \\
                                   1 & \cdots & 0 \\
                                 \end{array}
                               \right]$
that transforms the basis $(e_1, \ldots, e_d)$ in the basis $(e_d, \ldots, e_1)$. Then
\begin{eqnarray*}
 ||x_n^t \widetilde{b}\widetilde{\Lambda}_n||_p
 &=& ||x_n^t \widetilde{b}s s^{-1}\widetilde{\Lambda}_n s||_p  \qquad\mbox{since $\|s\|_p=\|s^{-1}\|_p=1$} \\
 &=& ||x_n^t \widetilde{b}s \Lambda_n^{-1} ||_p \qquad\mbox{since }s^{-1}\widetilde{\Lambda}_n s=\Lambda_n^{-1}
 \\&=& || \Lambda_n^{-1} (\widetilde{b}s)^t  x_n||_p\qquad\mbox{since }\|g^t\|_p=\|g\|_p.
\end{eqnarray*}
Set $\overline{b}=((\widetilde{b}s)^t)^{-1}$, then
$$\lim_{n\to +\infty} \frac{1}{n} \ln ||\Lambda_n^{-1} \overline{b}^{-1}  x_n||_p=0.$$

We want to show now that if $b$ is as in \eqref{eq-oselbis} then $b$ and $\overline{b}$ are in the same class in $B_p$.  To do this observe that \begin{equation}\label{eq-P-lambda}
u\in P_p \Longleftrightarrow \lim_{n\to +\infty}\frac{1}{n} \ln ||\Lambda^{-1}_n u \Lambda_n||_p = 0 \Longleftrightarrow \lim_{n\to +\infty}\frac{1}{n} \ln ||\Lambda^{-1}_n u \Lambda_n||_p \leq 0.
\end{equation}
This  can be proved by direct calculations using the fact that   $\max_{i,j}|g_{i,j}|_p\leq ||g||_p\leq d^2\max_{i,j}|g_{i,j}|_p$ for all $p\in \kP$.

Then since
$$\ln \|\Lambda_n^{-1}\overline{b}^{-1}b\Lambda_n\|_p=\ln \|\Lambda_n^{-1}\overline{b}^{-1}x_nx_n^{-1}b\Lambda_n\|_p \leq \ln \|\Lambda_n^{-1}\overline{b}^{-1}x_n\|_p + \ln \|x_n^{-1}b\Lambda_n\|_p,$$
it follows immediately that $\overline{b}^{-1}b\in P_p$.

On the other hand
$$||\Lambda_n^{-1} \overline{b}^{-1} x_n||_p||\Lambda_n^{-1} u\Lambda_n||_p^{-1}\leq
 ||\Lambda_n^{-1} u^{-1}\overline{b}^{-1} x_n||_p \leq
 ||\Lambda_n^{-1} u^{-1}\Lambda_n||_p||\Lambda_n^{-1} \overline{b}^{-1} x_n||_p.$$
Then for every $b_1=\overline{b}u$ with $u$ in the group $P_p$,
$$\lim_{n\to +\infty} \frac{1}{n} \ln ||\Lambda_n^{-1} b_1^{-1}  x_n||_p=0.$$
Thus for all $b\in \mathbf{b}(w)$,
$$\lim_{n\to +\infty} \frac{1}{n} d_p\left(x_n,b\Lambda_n\right)=\lim_{n\to +\infty} \frac{1}{n} (\ln^+ \|x_n^{-1}b\Lambda_n\|_p+\ln^+ ||\Lambda_n^{-1} b^{-1}  x_n||_p)=0.$$

It just remains to see that the class $\mathbf{b}(w)$ is the unique such that \eqref{eq-LLN} holds. But if $b_1$ and $b_2$ are two matrices such that \eqref{eq-LLN} holds, then
$$0=\lim_{n\to +\infty} \frac{1}{n} d_p\left(b_1\Lambda_n,b_2\Lambda_n\right)=\lim_{n\to +\infty} \frac{1}{n} d_p\left(e,\Lambda_n^{-1}b_1^{-1}b_2\Lambda_n\right),$$
and using once more \eqref{eq-P-lambda} we conclude.
\end{proof}

\subsection{The spaces $B_p$ and $\mathbb{B}$ are $\mu$-boundaries}
It is easily checked, using left-invariance of $d_p$,  that the function $\mathbf{b}_p$ defined in Proposition \ref{LLN} satisfies
the hypothesis of Proposition \ref{prop-mu-bound}. Then we immediately get 

\begin{cor}
Let $\mu$ be a probability measure on $GL_d(\Q_p)$. Assume that $\int d_p(e,g)\, d\mu(g)<+\infty$, and let $\nu_p$ be the law of $\mathbf{b}_p$. Then $(B_p,\nu_p)$ is a $\mu$-boundary.

Let $\mu$ be a probability measure on $GL_d(\Q)$. Assume that $\int d(e,g)\, d\mu(g)<+\infty$. Let
$\mathbf{b}$ be the map from $\Omega$ to $\mathbb{B}=\prod_{p\in\kP}B_p$
defined by:
$$\mathbf{b}: w\mapsto \mathbf{b}(w)=(\mathbf{b}_p(w))_{p\in\kP}.$$
Let $\nu$ be the law of $\mathbf{b}$. Then $(\mathbb{B},\nu)$ is a $\mu$-boundary.
\end{cor}

\section{Poisson boundary of $GL_d(\Q)$}
To prove that $\B$ is the maximal $\mu$-boundary, we use the following lemma, which is a generalization of the ray criterion of V. Kaimanovich \cite{Kai00}, already implicitly used in our previous works \cite{Bro,Sch}.

\begin{lem}\label{lem-ray}
Let $\mu$ be a probability measure on a countable group $G$ with finite entropy. Let $(B,\nu)$ be a $\mu$-boundary and $\mathbf{b}$ the associated boundary map. Suppose that for each $n$ there exits a measurable map $C_n$ from $B$ to subsets of $G$ such that:
$$\lim_{n\to+\infty}\pp( x_n\in C_n(\mathbf{b}(w)))=1 \mbox{ and } \lim_{n\to+\infty}\frac{1}{n}\ln |C_n(z)|\le \delta \quad \mbox{$\nu(dz)$-almost surely}. $$
Then $h^z\leq \delta$ for $\nu$-almost all $z$.
\end{lem}

\begin{proof}
Observe that
$$
\pp\left(x_{n}\in C_n(\textrm{\textbf{b}}(w))\right)=\int_{B}\pp_{n}^{z}\left[C_n(z)\right]\nu(dz)\to 1.$$
Thus, along a sub-sequence, $\pp_{n}^{z}\left[C_n(z)\right]$
converges to 1 for $\nu$-almost all $z$.

Recall that $h^z$ is the $\pp^{z}$-almost sure limit of $-\ln\pp_{n}^{z}(x_{n})/n$.
Now for any $\varepsilon>0$ consider the set
$$A_{n}(z)=\left\{ g\in G \mid -h^z-\varepsilon<\ln\pp_{n}^{z}(g)/n<-h^z+\varepsilon\right\} .$$
Then $\pp_{n}^{z}(A_{n}(z)\cap C_n(z))$
converges to 1 on a sub-sequence, while, for large $n$
$$\pp_{n}^{z}(A_{n}(z)\cap C_n(z))\leq e^{n\left(\varepsilon-h^z\right)}|C_n(z)|\leq e^{n\left(\varepsilon-h^z\right)}e^{n\left(\delta+\varepsilon\right)}.$$
Thus $\delta-h^z+2\varepsilon\geq 0$. Since
$\varepsilon$ was arbitrarily chosen, we get $h^z\leq \delta.$
\end{proof}

In order to apply this lemma in our setting we need to show that the gauge on $GL_d(\Q)$ associated to the distance $d$ grows at most exponentially:

\begin{lem}\label{lem-card}
For $g\in GL_d(\Q)$  and $R\geq 0$, let $$B(g,R)=\{h\in GL_d(\Q)\mid d(g,h)\leq
  R\}.$$
Then there exits a constant $C>0$ such that for all $g$ and $R$
$$|B(g,R)|\leq Ce^{CR}.$$
\end{lem}

\begin{proof}
First observe that since  $d(e,g^{-1}h)=d(g,h)$, we have $g^{-1}B(g,R)=B(e,R)$. Thus two balls with the same radius have the same cardinality, and we
can restrict us without loss of generality to the case $g=e$.

Observe now that if $h=(h_{i,j})\in B(e,R)$ then for all couples of indices $(i,j)$
$$\sum_{p\in\kP}\ln^+ |h_{i,j}|_p\leq\sum_{p\in\kP}\max_{i,j}\ln^+|h_{i,j}|_p\leq d(e,h)\leq R.$$
It can be shown (see for instance \cite{Bro}) that there exists $C'$ such that for all $R$
$$\left|\left\{q\in\Q \mid \sum_{p\in\kP}\ln^+ |q|_p<R \right\}\right|\leq C'e^{C'R}.$$
The desired result follows.
%Then $|B(g,R)|=|B(e,R)|\leq  C'^{d^2}e^{C'd^2R}$.
\end{proof}

\begin{proof}[Proof of Theorem \ref{theo1}]
First observe that, since $\mu$ has finite first moment with respect of
to an exponentially growing gauge, it has finite entropy.

For any $p$, consider the moment of the random walk with respect  to $d_p$:
$$m_p=\int  d_p(e,g) \ d\mu(g).$$
Observe that $\sum_{p\in \kP} m_p=\E(d(e,w_1))<+\infty$.
Fix $F$ a finite subset of $\kP$ and  set  $m_{F^c}=\sum_{p\in F^c} m_p$. By the law of large numbers, $\pp$-almost surely
$$\frac{\sum_{p\in F^c}d_p(x_n,e)}{n}
\leq \frac{\sum_{k=1}^n \sum_{p\in F^c} d_p(x_k,x_{k-1})}{n}
= \frac{\sum_{k=1}^n \sum_{p\in F^c} d_p(w_k,e)}{n}
\to m_{F^c}.$$

Fix $\varepsilon>0$ and
$\mathbf{b}=(b_p)_{p\in \kP}\in \B$, and set
$$C^{F,\varepsilon}_{n}(\mathbf{b})=
\left\{g\in GL_d(\Q) \mid d_p(g,b_p\Lambda_n(p))\le n\,\varepsilon \ \forall p\in F,\ \sum_{p\in F^c}d_p(g,e)\leq n \, (m_{F^c} +\varepsilon)  \right\}.$$

Then by Proposition \ref{LLN}
$$\pp\left[x_n\in C^{F,\varepsilon}_n(\mathbf{b}(w))\right]\to 1.$$

To apply Lemma \ref{lem-ray}, we need to control the cardinality of $C^{F,\varepsilon}_n$. Suppose that $C^{F,\varepsilon}_{n}(\mathbf{b})$ is nonempty and let $g_0\in C^{F,\varepsilon}_{n}(\mathbf{b})$. Then for all $g\in C^{F,\varepsilon}_{n}(\mathbf{b})$,
\begin{eqnarray*}
  d(g_0,g)&=& \sum_{p\in \kP} d_p(g_0,g)\\&\leq& \sum_{p\in F} \left(d_p(g_0,b_p\Lambda_n)+d_p(b_p\Lambda_n,g)\right)+ \sum_{p\in \kP-F}\left( d_p(g_0,e)+d_p(e,g)\right)\\
  &\leq& 2n\,(|F|\varepsilon + m_{F^c}+\varepsilon)
\end{eqnarray*}
Thus
$$\frac{1}{n}\ln|C^{F,\varepsilon}_{n}(\mathbf{b})|\leq \frac{1}{n}\ln |B(g_0, 2n\,(|F|\varepsilon + m_{F^c}+\varepsilon))|\leq 2nC\,(|F|\varepsilon + m_{F^c}+\varepsilon)+\frac{\ln C}{n}.$$
Thus for all finite $F$ and all $\varepsilon>0$,
$$h^z\leq 2C(|F|\varepsilon + m_{F^c} +\varepsilon).$$
Letting $\varepsilon$ go to zero and $F$ grow to $\kP$ (in such a way $m_{F^c}$ goes to zero), it follows that $h^z=0$ and
thus that $(\B, \nu)$ is the Poisson boundary.
\end{proof}

To conclude  we prove our triviality criterion:
\begin{proof}[Proof of Corollary \ref{cor-triv}]
It is immediate that if $\lambda_1(p)=\lambda_d(p)$ then $P_p=GL_d(\Q_p)$, thus $B_p$ is trivial.

Suppose now that the projection of $\nu$ on $B_p$  is trivial , but that $\lambda_1(p) > \lambda_d(p)$. In this case $B_p$ is nontrivial and the projection of $\nu$ on $B_p$ is a dirac measure whose mass is concentrated in a point $b\in B_p$ that is fixed by the support of $\mu$. Then the support of $\mu$ fixes all sub-spaces that compose the nontrivial flag associated to $b$. This contradicts the fact that no proper subspace of $\Q_p^d$ is fixed by the support of $\mu$.
\end{proof}

\end{document}